\documentclass[12pt]{article}
\usepackage{pifont}

\usepackage{amsmath}
\usepackage{amsfonts}
\usepackage{latexsym}
\usepackage{amssymb}
\usepackage{stmaryrd}
\usepackage{overpic}
\usepackage{graphicx,subfigure}

\usepackage{tikz}
\usepackage{hyperref}
\newtheorem{theorem}{Theorem}[section]

\newtheorem{lemma}[theorem]{Lemma}

\newenvironment{proof}[1][Proof]{\textbf{#1.} }
{\ \rule{0.75em}{0.75em}\smallskip}

\newenvironment{@abssec}[1]{%
     \if@twocolumn
       \section*{#1}%
     \else
       \vspace{.05in}\footnotesize
       \parindent .2in
         {\upshape\bfseries #1. }\ignorespaces
     \fi}
     {\if@twocolumn\else\par\vspace{.1in}\fi}

\newcommand{\TheTitle}{High-Order Extended Finite Element Methods for Solving Interface Problems}
\newcommand{\email}[1]{\protect\href{mailto:#1}{#1}}
\newcommand\keywordsname{Key words}
\newcommand\AMSname{AMS subject classifications}
\newenvironment{keywords}{\begin{@abssec}{\keywordsname}}{\end{@abssec}}
\newenvironment{AMS}{\begin{@abssec}{\AMSname}}{\end{@abssec}}
\title{{\TheTitle}\thanks{The first and third author is supported in part by the U.S. Department of Energy, Office of Science, Office of Advanced Scientific Computing Research as part of the Collaboratory on Mathematics for Mesoscopic Modeling of Materials under contract number DE-SC0009249. The second author is supported in part by the Fundamental Research Funds for the Central Universities under grant 1118020303 and China NSF under the grant 11101208.}}
\author{
  Fei Wang\thanks{Department of Mathematics, Pennsylvania State University, State College, PA
    (\email{feiwang.psu@gmail.com}, \url{http://www.personal.psu.edu/fuw7}).}
  \and
  Yuanming Xiao\thanks{Department of Mathematics, Nanjing University, Jiangsu, 210093, P.R. China. (\email{xym@nju.edu.cn}). }
  \and
  Jinchao Xu\thanks{Department of Mathematics, Pennsylvania State University, State College, PA
    (\email{jinchao@psu.edu}, \url{http://www.math.psu.edu/xu}).}
}

\begin{document}
\date{}
\maketitle
\begin{keywords}
Elliptic interface problems, unfitted mesh, extended finite element, high order
\end{keywords}

\begin{AMS}
65N12, 
65N15, 
65N30 
\end{AMS}

\begin{abstract}
In this paper, we study arbitrary order extended finite element (XFE) methods based on two discontinuous Galerkin (DG) schemes in order to solve elliptic interface problems in two and three dimensions. Optimal error estimates in the piecewise $H^1$-norm and in the $L^2$-norm are rigorously proved for both schemes. In particular, we have devised a new parameter-friendly DG-XFEM method, which means that no ``sufficiently large'' parameters are needed to ensure the optimal convergence of the scheme. To prove the stability of bilinear forms, we derive non-standard trace and inverse inequalities for high-order polynomials on curved sub-elements divided by the interface. All the estimates are independent of the location of the interface relative to the meshes. Numerical examples are given to support the theoretical results.

\end{abstract}

\section{Introduction}
Many multi-physics problems, including fluid--structure interaction problems and multiphase flow problems, involve coupling between different physical systems through the interface, which separates two phases of matter, i.e., solid, liquid, or gaseous. In the endeavor to solve such multi-physics problems, one of the most challenging tasks is that of devising an accurate numerical discretization of the interface problems. In this paper,
we consider the following elliptic interface problem:
\begin{equation} \label{interface_problem}
\left\{
\begin{array}{rccl}
-\nabla\cdot (\alpha(x) \nabla u)&=&f, &\mbox{ in } \Omega_1\cup\Omega_2, \\
 \left[\alpha(x) \nabla u\right] &=& g_N, &\mbox{ on } \Gamma,\\
\left[u\right]&=&\boldsymbol{g_D}, &\mbox{ on } \Gamma,\\
u&=&0, &\mbox{ on } \partial\Omega.
\end{array}
\right.
\end{equation}
The global regularity of the solution is low due to the nature of the interface. Here, domain $\Omega$ is a bounded and convex polygonal/polyhedral domain in $\mathbb{R}^d~(d=2$ or $3)$ and an internal interface $\Gamma$ divides $\Omega$ into two open sets, $\Omega_1$ and $\Omega_2$. We assume that $\Gamma= \partial\Omega_1$ is $C^2$-smooth (see Figure \ref{fig:domain} for an illustration of a unit square that contains a circle as an interface). We also assume that $\Gamma\cap\partial\Omega=\emptyset$. The jump $[\cdot]$ is defined in \eqref{jump_ave_scalar} and \eqref{jump_ave_vector}, and the coefficient $\alpha(x)$ is bounded from below and above by some positive constants. Due to the discontinuity of the coefficient $\alpha(x)$, the standard numerical methods, which are efficient for smooth solutions, usually lead to a loss of accuracy across the interface.

One way to render the more accurate approximation is to use interface-fitted/resolved grids. This way, the non-smoothness of the solution can be restricted to a ``narrow'' sub-domain in respect to the grid near the interface, such that the approximation error caused by the grid-mismatch is reduced to some extent. In \cite{xu82} (see also \cite{xu13} for an English translation) and more recently in \cite{chen98},  the following error estimate is obtained for $d=2$:
\begin{equation*}
\|u-u_h\|_{L^2(\Omega)}+h  |u-u_h|_{H^1(\Omega_1\cup\Omega_2)}\lesssim |\log h|^{1/2}h^2|u|_{H^2( \Omega_1\cup\Omega_2)}.
\end{equation*}
A sharper analysis is given in \cite{bramble96}, wherein the logarithm factor of the above estimate is removed for $d=2$. Here, we use the notation $H^m(\Omega_1\cup\Omega_2)= \{v\in L^2(\Omega), v|_{\Omega_1}\in H^m(\Omega_1)\ {\rm and}\ v|_{\Omega_2}\in H^m(\Omega_2)\}$, which is equipped with the norm $\|\cdot\|_{H^m(\Omega_1\cup\Omega_2)}=(\|\cdot\|^2_{H^m(\Omega_1)}+\|\cdot\|^2_{H^m(\Omega_2)})^{1/2}$. The interface-fitting assumption in the works referenced thus far can be loosened slightly so that the interface $\Gamma$ is ``$\mathcal{O}(h^2)$-resolved by the mesh" \cite{li10}. Further, the shape-regularity restriction of the grid can be loosened to maximal-angle-bounded grids \cite{chen15}. The optimal approximation in regard to the accuracy of the linear element space can also be proved on these grids.

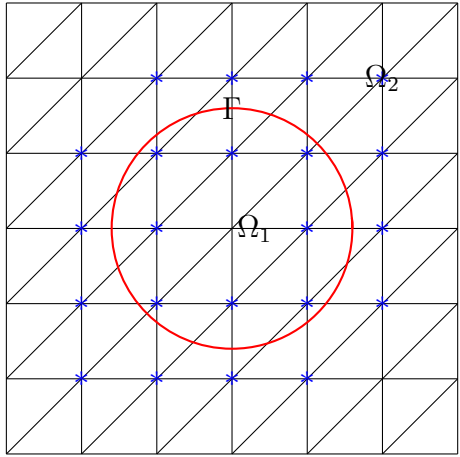
\begin{figure}[htbp]
\begin{center}

\begin{tikzpicture}[scale=1]
\draw(0,0)--(6,0);
\draw(0,1)--(6,1);
\draw(0,2)--(6,2);
\draw(0,3)--(6,3);
\draw(0,4)--(6,4);
\draw(0,5)--(6,5);
\draw(0,6)--(6,6);

\draw(0,0)--(0,6);
\draw(1,0)--(1,6);
\draw(2,0)--(2,6);
\draw(3,0)--(3,6);
\draw(4,0)--(4,6);
\draw(5,0)--(5,6);
\draw(6,0)--(6,6);

\draw(0,1)--(5,6);
\draw(0,2)--(4,6);
\draw(0,3)--(3,6);
\draw(0,4)--(2,6);
\draw(0,5)--(1,6);
\draw(0,0)--(6,6);
\draw(1,0)--(6,5);
\draw(2,0)--(6,4);
\draw(3,0)--(6,3);
\draw(4,0)--(6,2);
\draw(5,0)--(6,1);

\draw[thick, red] (4.6,3) arc (0:365:1.6);

\draw(1,1)[blue]node{$*$};
\draw(2,1)[blue]node{$*$};
\draw(3,1)[blue]node{$*$};
\draw(4,1)[blue]node{$*$};
\draw(1,2)[blue]node{$*$};
\draw(2,2)[blue]node{$*$};
\draw(3,2)[blue]node{$*$};
\draw(4,2)[blue]node{$*$};
\draw(5,2)[blue]node{$*$};
\draw(1,3)[blue]node{$*$};
\draw(2,3)[blue]node{$*$};
\draw(4,3)[blue]node{$*$};
\draw(5,3)[blue]node{$*$};
\draw(1,4)[blue]node{$*$};
\draw(2,4)[blue]node{$*$};
\draw(3,4)[blue]node{$*$};
\draw(4,4)[blue]node{$*$};
\draw(5,4)[blue]node{$*$};
\draw(2,5)[blue]node{$*$};
\draw(3,5)[blue]node{$*$};
\draw(4,5)[blue]node{$*$};
\draw(5,5)[blue]node{$*$};
\draw(5,5)node{$\Omega_2$};
\draw(3.3,3)node{$\Omega_1$};
\draw(3,4.6)node{$\Gamma$};

\end{tikzpicture}
\caption{Domain $\Omega=\Omega_1\cup\Gamma\cup\Omega_2$ with an unfitted mesh.}
\label{fig:domain}
\end{center}
\end{figure}

In an interface-fitted mesh, the sides ($d=2$) or  the edges ($d=3$) intersect with the interface only through their vertices. Unfortunately, it is usually a nontrivial and time-consuming task to construct good interface-fitted meshes for problems involving geometrically complicated interfaces. When the problem is time-dependent, the domain needs to be re-meshed at each time step, which introduces an interpolation error between two consecutive meshes. Therefore, numerous unfitted mesh methods, in which the interface is allowed to cross the elements, have been proposed in the literature. In the finite difference setting, we refer to the immersed boundary method in \cite{peskin77}, the immersed interface method in \cite{ll94, I06}, the ghost fluid method in \cite{lfk00}, and the references therein. In the finite element framework, we refer to the work of \cite{cgh08, gll08, llw03} for elliptic problems with discontinuous coefficients in which finite element basis functions are locally modified for elements that intersect with the interface where the coefficient jumps. In \cite{chen_xiao}, the adaptive immersed interface finite element method based on a posteriori error estimates is proposed for elliptic and Maxwell equations with discontinuous coefficients.

In the past decade, a combination of the extended finite element method (XFEM) (sometimes also known as the unfitted finite element method) with the Nitsche scheme has become a popular discretization method.
XFEM was introduced in the context of crack formation in structure mechanics, and one of its benefits is the ability to model discontinuities independent of the mesh structure \cite{belytschko99,moes99}. The idea of XFEM is to enrich the original finite element space by using specially designed basis functions that reflect the local features (discontinuity, singularity, boundary layer, etc.) of the problem. We refer to \cite{fries10} and the references therein for a historical account of XFEM. Inspired by the simple idea for handling Dirichlet boundary conditions described in \cite{nitsche70}, Hansbo \cite{hansbo02} applied Nitsche's method to reformulate the problem \eqref{interface_problem} in an XFE space. Hansbo \cite{hansbo02} proved that this Nitsche-XFEM can achieve the optimal convergence rate for the linear element, thereby generalizing the results in \cite{b70,barrett87}. In \cite{bh12}, the Nitsche-type weak boundary conditions are extended to a fictitious domain setting. A penalty term acting on the jumps of the gradients is added over the element faces, and optimal a priori error estimates are thereby derived for the linear element. In \cite{massjung12}, an unfitted symmetric interior penalty discontinuous Galerkin method for elliptic interface problems is considered, and optimal $h$-convergence for arbitrary $p$ is given for the two-dimensional case in the energy norm and in the $L^2$-norm. In \cite{wu2010}, an unfitted $hp$-interface penalty finite element method for elliptic interface problems is studied for both two and three dimensions. An extra flux penalty term is added to the bilinear forms. Thus the stability could be proved by applying local trace and inverse inequalities on regular sub-elements. In \cite{burman16}, the focus is a modified scheme for which the error estimates are independent of the contrast between diffusion coefficients. In \cite{wang15}, a quadratic Nitsche-XFEM is studied for the interface problem, and a clear classification of the shape of interface intersecting elements is given.
An overview of the ways in which Nitsche's method has been applied to interface problems is given in \cite{hansbo05}.

The Nitsche-XFEM can be interpreted as applying interior penalty (IP) methods on the interface, and techniques inspired by IP methods are used in \cite{hansbo02,massjung12,wu2010}. In this paper, we first extend IP-XFEM to high-order XFE spaces, and then we consider another new DG-XFEM for solving elliptic interface problems \eqref{interface_problem}.
We note that in our first approach, the penalization is applied only to the jump of the solution values across the interface (compared with the bilinear form in \cite{wu2010}), and the optimal $h$-convergence rate for arbitrary $p$ in the energy and $L_2$-norm are proved regardless of the dimension. The major and defining step in our variant is a delicate choice of the weight in the average (see \eqref{kappa}), which leads to an inverse estimate for possibly degenerated sub-elements (see \eqref{ineq_main}). Whereas Nitsche-type schemes are sometimes criticized for the inconvenient choice of stabilization parameters, we propose a second ``parameter-friendly'' scheme. In this scheme a penalization based on a lifting operator is introduced locally along the interface, and parameters need not be ``sufficiently large'' in regard to establishing the stability of the bilinear form. Furthermore, we derive a generalized version of C\' ea's lemma (see \eqref{inf-sup}) to retrieve the optimal convergence rate on high-order XFE spaces, even though the bilinear form is not bounded for functions in the continuous space in a normal sense. The main results of the analysis are summarized in Theorems \ref{thm1} and \ref{thml2}.

Note that the trace and inverse inequalities such as \eqref{ineq_main} are pivotal both in analyzing Nitsche-type methods, and in deriving approximate penalty parameters to stabilize these schemes.
They play an even more important role in the analysis of the unfitted mesh approach where the interface is allowed to intersect elements in an arbitrary manner. When sub-elements degenerate, which is not a rare case, the traditional technique by trace theorem and scaling argument is difficult to apply. The lowest version of \eqref{ineq_main} is derived in \cite{hansbo02}, which utilized the fact that the gradient of the linear polynomial is constant. A similar inequality has been proved for possibly degenerated sub-elements in \cite{massjung12} for two dimensions, whereas further justification is required for three dimensions. In Section \ref{sec:aux}, we prove the trace and inverse inequalities for polynomials of arbitrary order and for a general class of sub-elements, even though some of these may be very irregular in shape. The stability and the optimal convergence rate of the schemes are thus obtained.

This paper is organized as follows. In Section 2, we give some preliminary results, which are used in subsequent sections, and then we introduce the XFE spaces and reformulate the interface problem \eqref{interface_problem} in two types of DG schemes. In Section \ref{sec:aux}, we prove a special inequality \eqref{ineq_main} --- this is the key step in proving the stability of DG-XFEM with arbitrary polynomial order for both the 2-$d$ and 3-$d$ interface problems. The $H^{1}$- and $L^{2}$- error estimates of both schemes --- which attain the optimal order of the convergence rate in respect to mesh size $h$ --- are given in Section \ref{sec:error}. We also prove the parameter-friendly property of the second scheme in this section. Numerical examples are provided in Section \ref{sec:numerics} to support the theoretical results.

\section{XFE and DG schemes for interface problems}
\setcounter{equation}0

\subsection{Notation and XFE space}
We begin by providing some of the notation used in this paper.
Given a bounded domain $D\subset \mathbb{R}^d$ and a positive integer $m$, $H^m(D)$ is the Sobolev space with the corresponding usual norm and semi-norm, denoted, respectively, by $\|\cdot\|_{H^m(D)}$ and $|\cdot|_{H^m(D)}$. We use $|\cdot|$ for the measure of domains, such as the volume of a $3$-$d$ manifold, the area of a $2$-$d$ manifold, or the length of a $1$-$d$ manifold. In this paper, $d$ always denotes the dimension of domain $\Omega$, unless stated otherwise. Throughout the paper, ``$\lesssim\cdots $" stands for ``$\leq C\cdots $", the generic constant $C$ is independent of both mesh size $h$ and the location of the interface relative to the meshes.

Denote by $\{\mathcal{T}_h\}$, a
family of conforming, quasi-uniform, and regular partitions of $\Omega$ into triangles and parallelograms/tetrahedrons and parallelepipeds. For each element $K\in\mathcal{T}_h$, we use $h_K$ for its diameter. Let $h = \max\{h_K: K\in \mathcal{T}_h\}$. As $K$ is of regular shape, there is a constant $\gamma_0$ such that
\begin{align}\label{gamma0}
h_K^d\le\gamma_0 |K|, \quad \forall K\in \mathcal{T}_h.
\end{align}
We define the set of all elements intersected by $\Gamma$ as $\mathcal{T}_h^\Gamma = \{K\in \mathcal{T}_h: |K\cap\Gamma| \neq0\}$. For an element $K$ in $\mathcal{T}_h^\Gamma$, let $e_K=K\cap\Gamma$ be the part of $\Gamma$ in $K$. Each $\mathcal{T}_h$ induces a partition of interface $\Gamma$, which we denote by $\mathcal{E}_h^\Gamma = \{e_K: e_K=K\cap\Gamma, K\in \mathcal{T}_h^\Gamma\}$.
For any $K\in \mathcal{T}_h$, let $K_i = K\cap \Omega_i$ denote the part of $K$ in $\Omega_i$ and
$\boldsymbol{n}_i$ be the unit outward normal vector on
$\partial K_i$ with $i = 1,2$. As $\Gamma$ is of class $C^2$, it is easy to prove that (cf.\cite{chen98,xu13}) each interface segment/patch $e_K$ is contained in a strip of width $\delta$ and satisfies
\begin{align}\label{gamma1}
\delta\le \gamma_1 h_K^2~~~\mbox{and}~~~|\boldsymbol{n}_i(\mathrm{x})-\boldsymbol{n}_i(\mathrm{y})|\le \gamma_2 h_K,~\forall \mathrm{x},\mathrm{y}\in e_K.
\end{align}

Now, let us simply introduce the XFE space. Let $\chi_i$ be the characteristic function on $\Omega_i$ with $i=1,2$. Given a mesh $\mathcal{T}_h$, let $V_h$ be the continuous piecewise polynomial function space of degree $p\geq 1$ on the mesh. Let $V_h^1:= V_h\cdot \chi_1$ and $V_h^2 := V_h\cdot \chi_2$. Define the XFE space by $V_h^\Gamma = V_h^1 +V_h^2$. Note that the restrictions of the functions in $V_h^\Gamma$ in each sub-domain are standard continuous finite element functions, whereas discontinuity may occur only across $\Gamma$. Since the solution of problem \eqref{interface_problem} is non-smooth only in the vicinity of the interface, the XFE space is an appropriate choice for the discretization. Nitsche-XFEM, as noted in the introduction, can thus be regarded as relying on the application of the DG approach on the interface $\Gamma$ instead of on the interelement edges.

\subsection{DG schemes for interface problems}
For a scalar-valued function $v$, let $v_i$
= $v|_{\partial K_i}$, and similarly, for a vector-valued function $\boldsymbol{q}$,
we denote $\boldsymbol{q}_i=\boldsymbol{q}|_{\partial K_i}$. We define the weighted
average $\{\cdot\}$ and the jump $[\cdot]$ on $e\in \mathcal{E}_h^\Gamma$
by
\begin{align}
\{v\} &= \kappa_1 v_1 + \kappa_2 v_2,  &[v] = v_1 \boldsymbol{n}_1
   + v_2 \boldsymbol{n}_2,\quad\;\; \label{jump_ave_scalar}\\
\{\boldsymbol{q}\} &= \kappa_1\boldsymbol{q}_1 + \kappa_2\boldsymbol{q}_2, & [\boldsymbol{q}] = \boldsymbol{q}_1\cdot \boldsymbol{n}_1 +
\boldsymbol{q}_2\cdot \boldsymbol{n}_2. \label{jump_ave_vector}
\end{align}
For the stability analysis of our schemes, we define $(\kappa_1, \kappa_2)$ on each element as follows:
\begin{equation}\label{kappa}
\kappa_i = \left\{
\begin{array}{cl}
  0, &\mbox{ if } \frac{|K_i|}{|K|}<c_0 h_K, \\
  1, &\mbox{ if } \frac{|K_i|}{|K|}>1-c_0 h_K,\\
  \frac{|K_i|}{|K|}, &\mbox{ otherwise }.
\end{array}
\right.
\end{equation}
Clearly, $0\le\kappa_i\le1$ and $\kappa_1+\kappa_2=1$ so that $\{\cdot\}$ is a convex combination along $\Gamma$. Roughly speaking, we adopt the weight $\kappa_i=\frac{|K_i|}{|K|}$ suggested in \cite{hansbo02} for general sub-elements and we set $\kappa_i=0$ for $|K_i|<c h_K^{d+1}$. Actually, we expect that the contributions of functions with ``very small'' support, say, $O(h_K^{d+1})$, can be eliminated without influencing the approximation quality significantly. Here, the user-defined constant $c_0\ge2\gamma_0\gamma_1$ represents this threshold and $\gamma_0$, $\gamma_1$ are constants defined in \eqref{gamma0} and \eqref{gamma1}, respectively. In Lemma \ref{ltrace}, we already elaborate the dependence of $c_0$ on these generic constants. For an alternative definition of $\kappa_i$, we refer to \cite{hansbo05} and the remarks presented after the proof of Lemma \ref{ltrace} in Section \ref{sec:aux}.

For any scalar-valued function $v$ and any vector-valued function $\boldsymbol{w}$, we have the following identity:
\begin{equation*}
 (v_1 \boldsymbol{n}_1)\cdot \boldsymbol{w}_1 + (v_2 \boldsymbol{n}_2)\cdot \boldsymbol{w}_2 =  [v] \cdot \{\boldsymbol{w}\} +  \{v\} [\boldsymbol{w}] + (\kappa_2-\kappa_1)(v_1-v_2)[\boldsymbol{w}].
\end{equation*}
Testing the elliptic problem \eqref{interface_problem} by any $v\in V_h^\Gamma$, using integration by parts and the above identity, we have
 \begin{align}\label{consistency}
    &\int_{\Omega_1\cup\Omega_2} \alpha(x) \nabla u\cdot \nabla v - \int_{\Gamma} \{\alpha(x) \nabla u\}\cdot [v]=  \int_{\Omega} f\  v+\int_{\Gamma} g_N (\kappa_1v_2+\kappa_2v_1).
\end{align}

We propose two types of DG schemes for interface problem \eqref{interface_problem} on the XFE space $V_h^\Gamma$. As the restrictions of the functions in $V_h^\Gamma$ on each $\Omega_i$ are standard continuous finite element functions, we introduce penalty terms for only those elements cut by the interface in our bilinear forms.

The first scheme is inspired by the interior penalty (IP) methods. Let $V = H^2(\Omega_1\cup\Omega_2)$ and $V(h) = V_h^\Gamma + V$. We define a bilinear form on $V(h)\times V(h)$:
\begin{align}
B_h^{(1)}(w,v):=&\int_{\Omega_1\cup\Omega_2}\alpha(x) \nabla w \cdot \nabla v
-\int_{\Gamma}\{\alpha(x)\nabla w\}\cdot [v] \nonumber\\
&-\beta \int_{\Gamma}[w] \cdot \{\alpha(x)\nabla v\}
 +\sum_{K\in\mathcal{T}_h^\Gamma}\frac{\eta_{\beta}}{h_K}\int_{K\cap\Gamma}[w]\cdot [v],\label{IP_blinear_first}
\end{align}
where $\eta_{\beta} h_K^{-1}\int_{K\cap\Gamma}[w]\cdot [v]$ is a penalty term acting on each segment/patch of $\Gamma$ and $\eta_\beta$ is a parameter to be specified in Section \ref{sec:error}. Here, $\beta$ is a real number. When $\beta=1$,  $B_h^{(1)}(\cdot,\cdot)$ is symmetric and corresponds to the symmetric interior penalty Galerkin (SIPG) method \cite{arnold82,w78}, whereas $\beta =-1$ gives a non-symmetric interior penalty Galerkin (NIPG) formulation \cite{rwg99}.

To introduce the second type of penalization, for any $K\in \mathcal{T}_h^\Gamma$ and $e=K\cap\Gamma$, we define a lifting operator $r_e :[L^2(e)]^d \rightarrow W_K$:
\begin{align}\label{liftop}
\int_K r_e(\boldsymbol{q})\cdot \alpha(x)\boldsymbol{w}_h  &= -\int_{e} \boldsymbol{q}\cdot \{\alpha(x)\boldsymbol{w}_h\}, \;\; \forall\, \boldsymbol{w}_h \in W_K,
\end{align}
where $$ W_K=\{\boldsymbol{w}_h\in[L^2(\Omega)]^d:\; \boldsymbol{w}_h |_{K_i}\in [P_p(K_i)]^d,~i=1,2 ~~\mbox{and}~~ \boldsymbol{w}_h |_{\Omega\backslash K}=0\}.$$

By adding a penalization based on the operator $r_e$, we propose a parameter-friendly DG scheme that guarantees stability independent of a condition on the stabilization parameter:
\begin{align*}
B_h^{(2)}(w,v):=&\int_{\Omega_1\cup\Omega_2}\alpha(x) \nabla w \cdot \nabla v
-\int_{\Gamma}\{\alpha(x)\nabla w\}\cdot [v]
-\int_{\Gamma}[w] \cdot \{\alpha(x)\nabla v\}\\
&+\sum_{K\in\mathcal{T}_h^\Gamma}\frac{\eta_1}{h_K}\int_{K\cap\Gamma}[w]\cdot [v]+\sum_{e\in\mathcal{E}_h^\Gamma}\int_{\Omega}\eta \alpha(x)r_e([w])\cdot r_e([v]),
\end{align*}
where $\eta_1$ and $\eta$ are two positive parameters. Unlike $B_h^{(1)}(\cdot,\cdot)$, in which the selection of $\eta_\beta$ depends on the geometric property of the interface and triangulation, we prove that the scheme \eqref{dg_primal} based on $B_h^{(2)}(\cdot,\cdot)$ has a parameter-friendly feature. Further, the sparsity of the stiffness matrix is not affected. In fact, the only requirement for the well-posedness of \eqref{dg_primal} is $\eta_1 \geq 1$ and $\eta \geq 2$.

Define further the linear form $F_h^{(i)}(\cdot), i = 1, 2$ on $V(h)$:
\begin{align}
 F_h^{(1)}(v):=&\int_{\Omega} f\  v+\int_{\Gamma} g_N (\kappa_1v_2+\kappa_2v_1)-\beta\int_{\Gamma} \boldsymbol{g_D}\cdot \{\alpha(x)\nabla v\}+\sum_{K\in\mathcal{T}_h^\Gamma}\frac{\eta_\beta}{h_K}\int_{K\cap\Gamma}\boldsymbol{g_D}\cdot [v],\label{scheme1}\\
 F_h^{(2)}(v):=&F_h^{(1)}(v)+\sum_{e\in\mathcal{E}_h^\Gamma}\int_{\Omega}\eta \alpha(x)r_e(\boldsymbol{g_D})\cdot r_e([v]),\qquad \mbox{with}~~ \beta=1.\label{scheme2}
\end{align}
Then, the DG-XFE method for the interface problem \eqref{interface_problem} is: Find $u_h\in V_h^\Gamma$ such that
\begin{align}\label{dg_primal}
B_h(u_h,v_h) &= F_h(v_h), \quad \forall v_h\in V_h^\Gamma,
\end{align}
where $B_h(\cdot,\cdot) = B_h^{(i)} (\cdot,\cdot)$ and $F_h(\cdot)=F_h^{(i)}(\cdot)$ with $i = 1, 2$.

As the solution $u$ of \eqref{interface_problem} satisfies \eqref{consistency}, it is easy to check that \eqref{dg_primal} has become an identity for both schemes if we replace $u_h$ with $u$. Furthermore, the Galerkin orthogonality holds true:
\begin{equation}\label{dg_ortho}
B_h(u - u_h,v_h) = 0, \quad\forall\, v_h\in V_h^\Gamma.
\end{equation}

\subsection{Norm-equivalence property}
To end this section, we derive a norm-equivalence result (Lemma \ref{lemma:homothety}) that relates the $L^2$-norm of any polynomial functions in a bounded convex domain to the $L^2$-norm in a subset of comparable size. This property consists of the main step toward the proof of the trace and inverse inequalities in the next section. We start from a variant result of norm-equivalence in finite dimensional spaces.
\begin{lemma}\label{pullback0}
Given an integer $p\geq 0$ and $\lambda\in(0,1)$. For any $v(x)\in P_p[0,1]$, there exists a constant $C$ dependent only on $\lambda$ and $p$ such that
\begin{align}\label{polyl2}
\|v\|_{L^2(0,1)} \le C(\lambda,p)\|v\|_{L^2(0, \lambda)},\qquad \|x^{\frac12}v\|_{L^2(0,1)} \le C(\lambda,p)\|x^{\frac12}v\|_{L^2(0, \lambda)}.
\end{align}
\end{lemma}
\begin{figure}[htbp]
\begin{center}
\includegraphics[scale=1.2]{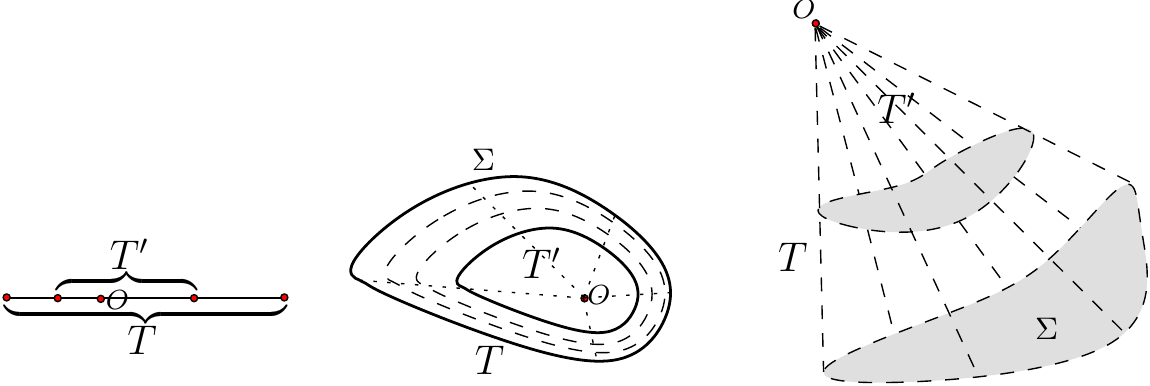}
\caption{The $L^2$-norm in the whole convex domain $T$ is dominated by the $L^2$-norm in $T'$, a subset of $T$, for any $v\in P_p(T)$.}
\label{fig:homothety}
\end{center}
\end{figure}
For any domain $T\in \mathbb{R}^d$, we say that $T_0$ is a homothetic image of $T$ if $T_0= \{\lambda (x-x_0)+x_0:\, x\in T\}$ for suitable $\lambda>0$ and $x_0\in\mathbb{R}^d$. Here, $x_0$ is called the homothetic center, from which each point $x$ in $T$ is mapped to a corresponding $x'$ in $T_0$ on the ray $\overrightarrow{x_0x}$ such that $\overrightarrow{x_0x'}=\lambda \overrightarrow{x_0x}$.
\begin{lemma}\label{lemma:homothety}
Given an integer $p\geq 0$ and $\lambda\in(0,1)$. Let $T$ be a closed convex domain in $\mathbb{R}^d$ with a (piecewise) smooth boundary. Assume that $T'$ contains a homothetic subset of $T$ with the scaling factor $\lambda$. Then, for any $v\in P_p(T)$, we have
\begin{align}\label{homothety}
\|v\|_{L^2(T)} \le C(\lambda,p+1)\|v\|_{L^2(T')},
\end{align}
where the upperbound constant $C(\lambda,p+1)$ is inherited from \eqref{polyl2}.
\end{lemma}
\begin{proof}
We only need to consider case $T'$ as a homothetic subset of $T$. As $\lambda<1$, by the fixed-point theorem, the homothetic center $O\in T'$ (Figure \ref{fig:homothety}). Without loss of generality, we take $O$ as the origin such that $T$ can be seen as a continuous contraction from (part of) its boundary $\Sigma$ (Figure \ref{fig:homothety}), that is,
$$T=\{x:\; x=s\cdot\mathbf{r},\;s\in [0,1],\; \mathbf{r}\in \Sigma\}, \quad\mbox{and}\quad T'=\lambda T.$$

For $d=1$, the result of \eqref{homothety} is a direct consequence of \eqref{polyl2} by applying the scaling argument on each segment of $T$ separated by $O$. For higher dimensions, the result can be derived by reducing a multiple integral to single integrals.

For example, when $d=2$, let $\Sigma$ be parameterized by $x=\mathbf{r}(\xi), \xi\in I$. Note that $s|\mathbf{r}(\xi)\times \mathbf{r}'(\xi)|$ is the absolute value of  the Jacobian determinant of the mapping $x =s\cdot\mathbf{r}(\xi)$. We can rewrite the double integrals of $v^2$ in a $(s, \xi)$-coordinate system and thereby obtain
\begin{align*}
\|v\|_{L^2(T)}^2=&\int_{I}|\mathbf{r}(\xi)\times \mathbf{r}'(\xi)|\mathrm{d}\xi \int_0^1v^2(s\cdot\mathbf{r(\xi)})s\mathrm{d}s\\
\le& C^2(\lambda,p)\int_{I}|\mathbf{r}(\xi)\times \mathbf{r}'(\xi)|\mathrm{d}\xi \int_0^{\lambda}v^2(s\cdot\mathbf{r(\xi)})s\mathrm{d}s=C^2(\lambda, p)\|v\|_{L^2(T')}^2,
\end{align*}
where we used the fact of \eqref{polyl2} as $v^2(s\cdot\mathbf{r})$ is a 1-$d$ polynomial of $s$ for any given $\xi$. Similarly, if $d=3$, characterized $\Sigma$ by $x=\mathbf{r}(\xi,\eta), (\xi,\eta)\in U$, then we have
\begin{align*}
\|v\|_{L^2(T)}^2=&\int_{U}|\mathbf{r}(\xi,\eta)\cdot(\mathbf{r}_\xi(\xi,\eta)\times \mathbf{r}_\eta(\xi,\eta))|\mathrm{d}\xi \mathrm{d}\eta \int_0^1v^2(s\cdot\mathbf{r(\xi,\eta)})s^2\mathrm{d}s\\
\le& C^2(\lambda,p+1)\int_{U}|\mathbf{r}(\xi,\eta)\cdot(\mathbf{r}_\xi(\xi,\eta)\times \mathbf{r}_\eta(\xi,\eta))|\mathrm{d}\xi \mathrm{d}\eta \int_0^{\lambda}v^2(s\cdot\mathbf{r(\xi,\eta)})s^2\mathrm{d}s\\
=&C^2(\lambda,p+1)\|v\|_{L^2(T')}^2.
\end{align*}
This completes the proof of Lemma \ref{lemma:homothety}.
\end{proof}

\section{Special trace and inverse inequalities}\label{sec:aux}
\setcounter{equation}0

In this section, we give some special trace and inverse inequalities, which are important in the stability analysis of DG-XFEMs \eqref{dg_primal} for interface problems.

\begin{lemma}\label{aniso}
For any $v\in H^1(K_i)$, the following trace inequality holds:
\begin{equation}
\|v\|_{L^2(e_K)}^2 \lesssim  \|v\|_{L^2(K_i)}|v|_{H^1(K_i)}  + \int_{\partial K_i\backslash e_K}  v^2(s)\label{trace2}
\end{equation}
if $h\in (0, h_0]$. Here, $e_K=K\cap \Gamma$ and $h_0$ is a constant independent of the location of $\Gamma$ relative to $K$. In fact, we can make $h_0$ explicit with $h_0=\frac1{\gamma_2}$ where $\gamma_2$ is defined in \eqref{gamma1}.
\end{lemma}
\begin{proof}
Let $\Gamma_K$ be a line/plane passing at least $d$ points in $e_K$. Denote $\boldsymbol{n}$ as the unit outward normal vector to $\Gamma_K$. Then, we have
\begin{align*}
\int_{K_i} 2v \frac{\partial v}{\partial \boldsymbol{n}}  = \int_{K_i} \frac{\partial v^2}{\partial \boldsymbol{n}} = \int_{\partial K_i} v^2 \boldsymbol{n}\cdot\boldsymbol{n_i}
= & \int_{e_K}  v^2 \boldsymbol{n}\cdot\boldsymbol{n_i}   + \int_{\partial K_i\backslash e_K} v^2 \boldsymbol{n}\cdot\boldsymbol{n_i}.
\end{align*}
Based on the assumptions that $\Gamma$ is $C^2$ smooth and that mesh size $h$ is small enough (say, $h\le\frac1{\gamma_2}$, see \eqref{gamma1}), we have $\frac12\le\boldsymbol{n}\cdot\boldsymbol{n_i}\le1$ on $e_K$. It follows that
\begin{align*}
 \int_{e_K}  v^2
\leq 2\left(\int_{K_i} 2v \frac{\partial v}{\partial \boldsymbol{n}}-\int_{\partial K_i\backslash e_K} v^2 \boldsymbol{n}\cdot\boldsymbol{n_i} \right)\leq 4 \left( \|v\|_{L^2(K_i)}|v|_{H^1(K_i)}  + \int_{\partial K_i\backslash e_K}  v^2(s) \right),
\end{align*}
which completes the proof of Lemma \ref{aniso}.
\end{proof}

The estimate of the interpolation error along $\Gamma$ relies on the following variant of trace inequality, which is a corollary of the above lemma. We also refer to \cite{hansbo02, wu2010} for details of the proof.
\begin{lemma}\label{ltrace2} There exists a constant $C$ that is dependent on $\Gamma$ but independent of the relative position of $\Gamma$ to the mesh, such that for any interface segment/patch $e_K=K\cap \Gamma\in \mathcal{E}_h^\Gamma$,
\begin{equation}\label{trace}
\|v\|_{L^2(e_K)}^2 \leq C (h_K^{-1}\|v\|_{L^2(K)}^2 + h_K\|\nabla v\|_{L^2(K)}^2),  \quad \forall v\in H^1(K).
\end{equation}
\end{lemma}

In the following lemma, we derive trace and inverse inequalities on arbitrary convex domains in $\mathbb{R}^d$. We are not aware of any study in which the same or similar results relating to high-order polynomial functions are reported.
\begin{lemma} \label{lem:covex}
For any convex domain $T\subset\mathbb{R}^d$ with a (piecewise) smooth boundary and $v\in \mathcal{P}_p(T)$, the
following estimates hold:
\begin{align}
\|\nabla v\|_{L^2(T)} &\lesssim\frac{1}{r} \| v\|_{L^2(T)},\label{convex1}\\
\|v\|_{L^2(\partial T)}&\lesssim\frac{1}{r^{1/2}}\|v\|_{L^2(T)},\label{convex2}
\end{align}
where $r$ is the radius of the largest inscribed ball of $T$. Here, the hidden constants in the inequalities depend only on $p$ and $d$ and are independent of the shape of $T$.
\end{lemma}
\begin{proof}
It was shown in \cite{gruber83,schwarzkopf98} that for any convex body $T\subset\mathbb{R}^d$, there exists a homothetic pair of boxes $B_1$ and $B_2$ such that $B_1\supseteq T\supseteq B_2$. Here, by ``box'' we mean a parallelepiped generated by $d$ orthogonal vectors. Furthermore, if we take the homothetic center as the origin such that $B_2=\lambda B_1$, then $\lambda$ is uniformly bounded from below in terms of $d$. By the scaling argument on the boxes and Lemma \ref{lemma:homothety}, we have
$$
\|\nabla v\|_{L^2(T)} \leq \|\nabla v\|_{L^2(B_1)} \lesssim \frac{1}{r} \|v\|_{L^2(B_1)}  \lesssim \frac{1}{r} \|v\|_{L^2(B_2)}\le \frac{1}{r} \|v\|_{L^2(T)},\qquad \forall v\in P_p(B_1),
$$
which gives the result of \eqref{convex1}.

Concerning the second inequality, we perform an analysis only for 2-$d$ convex domains. A similar argument can be made for 3-$d$ convex bodies following the guideline for Lemma \ref{lemma:homothety}. Let $\partial T$ be parameterized (piecewise) by $x=\mathbf{r}(\xi), \xi\in I$, and let $P$ be the center of the largest inscribed circle in $T$ (Figure \ref{fig:convexdomain}).
\begin{figure}
\centering\includegraphics[scale=.9]{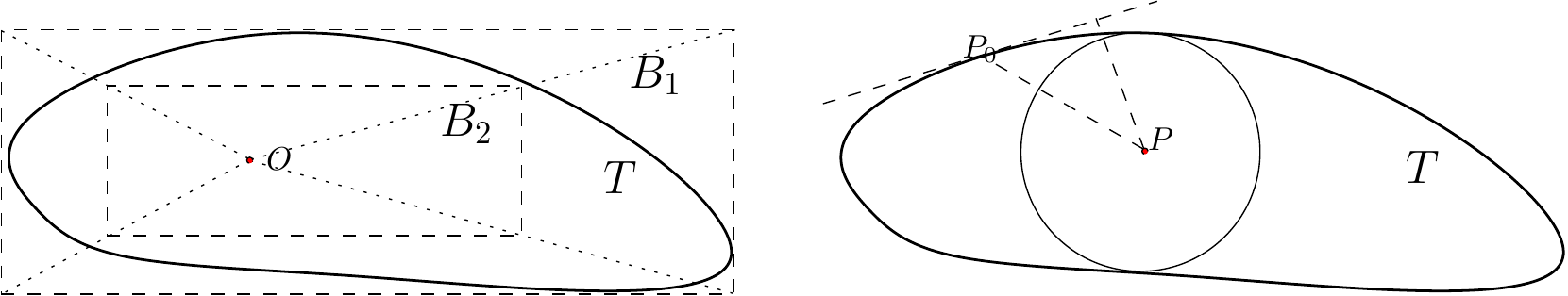}
\caption{$T$ is convex.}
\label{fig:convexdomain}
\end{figure}
$P$ is set as the origin, and $T$ is characterized by
$$T=\{x:\; x=s\cdot\mathbf{r}(\xi), \;s\in [0,1], \xi\in I\}.$$
Take $v\in \mathcal{P}_p(T)$ and consider its restriction on $\partial T$
as follows:
\begin{eqnarray} v^2({\bf r}(\xi))=
\int_0^1\frac{\partial}{\partial s}\big((s^2v^2(s\cdot {\bf r}(\xi))\big)\,\mathrm{d}
s=  2\int_0^1 sv(sv)_s\,\mathrm{d} s.\label{aaa}
\end{eqnarray}
Note that for any fixed $\xi$, $v$ is a polynomial whose degree does not exceeding $p$ with respect to $s$. Applying the following inverse inequality of 1-$d$ polynomial in $s$ (cf. \cite{schwab98}):
$$\|w'\|_{L^2([0,1])}\lesssim p^2\|w\|_{L^2([0,1])},\quad\forall w\in \mathcal{P}_p([0,1]),$$
we have
$v^2\lesssim p^2\int_0^1 v^2 s^2\,\mathrm{d} s \leq p^2\int_0^1 v^2 s\,\mathrm{d} s$. By integrating $v^2$ along $\partial T$, we find that
\begin{align*}
\|v\|_{L^2(\partial T)}^2&=\int_{I} v^2 |\mathbf{r}'(\xi)|\,\mathrm{d} \xi
\lesssim p^2\int_{I} \int_0^1 \,v^2(s\cdot \mathbf{r}(\xi))\, s\,|\mathbf{r}'(\xi)| \mathrm{d} s \mathrm{d} \xi\\
&\le p^2\sup_{\xi\in I}\frac{|\mathbf{r}'(\xi)|}{|\mathbf{r}(\xi)\times\mathbf{r}'(\xi)|}\|v\|_{L^2(T)}^2.
\end{align*}

We observe that $G(\xi):=\frac{|\mathbf{r}(\xi)\times\mathbf{r}'(\xi)|}{|\mathbf{r}'(\xi)|}$ corresponds to the distance from $P$ to the tangent line of $\partial T$ at point $P_0 = \mathbf{r(\xi)}$. As $T$ is convex, it always resides on one side of the tangent line. Therefore, $G(\xi)\ge r$ for any $\xi\in I$. Then, we derive that
$$\|v\|_{L^2(\partial T)}^2\lesssim \frac{p^2}{\inf_{\xi\in I}G(\xi)}
\|v\|_{L^2(T)}^2\le \frac{p^2}{r}\|v\|_{L^2(T)}^2,$$
which yields the conclusion of \eqref{convex2}.
\end{proof}

The crucial component in regard to establishing the stability of bilinear forms is the control on the weighted normal derivatives, which we state as a trace and inverse inequality in the following lemma. In \cite{hansbo02}, the validity of this inequality for $p=1$ leads to a heuristic choice of weight $\kappa_i=\frac{|K_i|}{|K|}$. Based on a slight modification of $\kappa_i$ defined in \eqref{kappa}, we extend the result to arbitrary polynomial degree $p$.

\begin{lemma}\label{ltrace}
Let $\gamma_0$ and $\gamma_1$ be constants defined in \eqref{gamma0} and \eqref{gamma1}, respectively. If we choose $c_0\ge 2\gamma_0\gamma_1$ in the definition \eqref{kappa} of $\kappa$, there exists a positive constant $h_0$ such that for all $h\in (0, h_0]$ and any interface segment/patch $e_K=K\cap \Gamma\in \mathcal{E}_h^\Gamma$, the following estimates hold on both sub-elements of $K$:
\begin{align}\label{ineq_main}
\|\kappa_i^{1/2}v_i\|_{L^2(e_K)} \leq \frac{C}{h_K^{1/2}} \|v_i\|_{L^2(K_i)},~~~~v_i\in \mathcal{P}_p(K_i),~~ i=1,2.
\end{align}
\end{lemma}
\begin{proof}
By the definition of the weight \eqref{kappa}, when $\frac{|K_i|}{|K|}<c_0 h_K$ or $\frac{|K_i|}{|K|}>1-c_0 h_K$, the result is either trivial or is reduced to a standard inverse inequality \cite{wu2010}. Thus, we need consider only the case where $\frac{|K_i|}{|K|}$ is bounded from below and above by $c_0 h_K$ and $1-c_0 h_K$.

\begin{figure}[htp]
\begin{center}
\includegraphics[scale=.8]{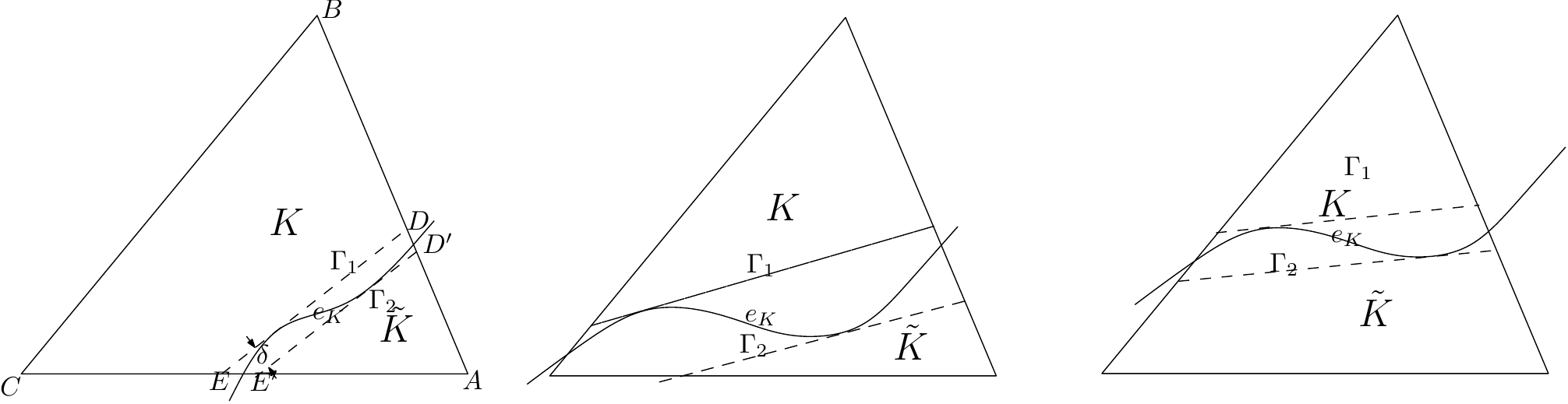}
\caption{A $2$-$d$ simplex intersected by $\Gamma$. $e_K=K\cap\Gamma$.}
\label{fig:simplex2d}
\end{center}
\end{figure}
\begin{figure}[htp]
\centering \subfigure[The intersection is a curved triangle.]{
\includegraphics[scale=.8]{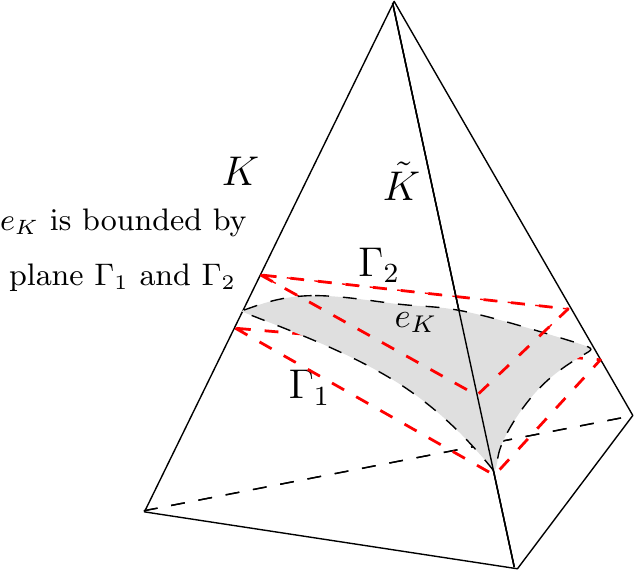}}\hspace{0.5in}
\subfigure[The intersection is a curved quadrilateral.]{
\includegraphics[scale=.8]{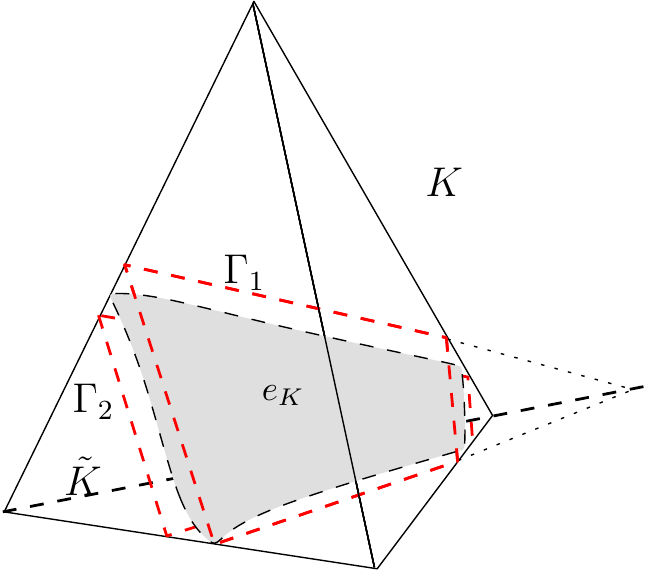}}
\caption{Intersection of $\Gamma$ with a $3$-$d$ simplex $K$. $\tilde K=K_1$ or $K_2$.}\label{fig:simplex3d}
\end{figure}

We recall that each interface segment/patch $e_K=K\cap\Gamma$ is contained in a strip of width $\delta$, which is not greater than $\gamma_0 h_K^2$.
Denote by $\Gamma_1$ and $\Gamma_2$ the two boundaries of the strip, which are parallel to a line/plane passing at least $d$ distinct points in $e_K$ (Figure \ref{fig:simplex2d}). Let $\tilde K=K_1$ or $K_2$ be a sub-element included in $K$. Each $\Gamma_i$ ($i=1,2$) divides $K$ into two polytopes. In these four polytopes, $T_1$ denotes the one includes $\tilde K$ and $T_2$ denotes the one included in $\tilde K$  (Figure \ref{fig:simplex2d} shows a $2$-$d$ example where $\tilde K$ is the sub-element bounded by $\Gamma$, $AB$, and $AC$, and we set $\Gamma_1=DE$,  $\Gamma_2=D'E'$, $T_1=\triangle ADE$, and $T_2=\triangle AD'E'$).

We know that the area/volume of $T_1$ can be expressed as the integration of the length/area of cross-sections along any given direction $\tau$. Take $\tau$ to be the normal vector to $\Gamma_1$. Let $d_{\Gamma_1}$ be the maximum distance from points in $T_1$ to $\Gamma_1$, and let $\delta$ denote the distance between $\Gamma_1$ and $\Gamma_2$. The measure of each cross section is less than $h^{d-1}$. Therefore, if $d_{\Gamma_1}< 2\delta$, we have that
$$|\tilde K|\le |T_1|\le d_{\Gamma_1}h_K^{d-1}< 2\delta h_K^{d-1}\le 2\gamma_0 h_K^{d+1}\le 2\gamma_0\gamma_1 h_K|K|\le c_0 h_K|K|.$$
In other words, the condition $\frac{|\tilde K|}{|K|}\ge c_0 h_K$ implies that $d_{\Gamma_1}\ge 2\delta$. That is, we need to justify \eqref{ineq_main} under this condition.
Suppose that $d_{\Gamma_1}$ is achieved at $P_0$, and let
$$T_0= \left\{\frac12 (P+P_0):\, P\in T_1\right\}.$$
On this basis, it is easy to verify that $T_0$ is included in $T_2$ when $d_{\Gamma_1}\ge 2\delta$. As $T_0$ is a homothetic copy of $T_1$ with a scaling factor of $\lambda=1/2$ and a homothetic center $P_0$, by Lemma \ref{lemma:homothety}, we have
\begin{align}\label{pullback2}
\|v\|_{L^2(T_1)}  \le C(1/2,p)\|v\|_{L^2(T_0)}\le C(1/2,p)\|v\|_{L^2(\tilde K)}, \quad \forall v\in \mathcal{P}_p(T_1).
\end{align}
In Figure \ref{fig:simplex2d}-\ref{fig:simplex3d}, we illustrate various intersections when a simplex $K$ is cut by a $(d-1)$-dimensional manifold $\Gamma$.

For any $v\in \mathcal{P}_p(\tilde K)$, we simply apply \eqref{trace2} and \eqref{pullback2} to obtain
\begin{align*}
\|v\|_{L^2(e_K)}^2  \lesssim &\|v\|_{L^2(\tilde K)}\|\nabla v\|_{L^2(\tilde K)}+\|v\|_{L^2(\partial \tilde K\backslash e_K)}^2\\
\le &\|v\|_{L^2(T_1)}\|\nabla v\|_{L^2(T_1)}+\|v\|_{L^2(\partial T_1)}^2\\
\lesssim &\frac{1}{r}\|v\|_{L^2(T_0)}^2\le \frac{1}{r}\|v\|_{L^2(\tilde K)}^2,
\end{align*}
where $r$ is the radius of the largest ball inscribed in $T_1$ and
\eqref{convex1} and \eqref{convex2} are used on $T_1$ in the last inequality. Since $|\tilde K|\le|T_1|\lesssim rh_K^{d-1}$, we obtain
\begin{align*}
\|\tilde\kappa^{1/2} v\|_{L^2(e_K)}  \lesssim & \left(\frac{|\tilde K|}{r|K|}\right)^{1/2}\|v\|_{L^2(\tilde K)}\lesssim \left(\frac{rh_K^{d-1}}{r|K|}\right)^{1/2}\|v\|_{L^2(\tilde K)}\lesssim h_K^{-\frac12}\|v\|_{L^2(\tilde K)},
\end{align*}
which completes the proof of Lemma \ref{ltrace}.
\end{proof}

The key step in the above proof is, roughly speaking, that of converting the original argument of \eqref{ineq_main} in $K_i$ to a variant in a possibly convex domain $T_0$ (with a straight/planar boundary) included in $K_i$. The estimate on $T_0$ is relatively easy to obtain with the help of Lemma \ref{lem:covex}.

We note that as an alternative, one simple choice of the element-wise defined average is to adopt $\kappa_i = 1$ if $|K_i|> \frac12|K|$ and $\kappa_i = 0$ if $|K_i|< \frac12|K|$. Thus, for an intersected element, we compute the numerical quantity only on the larger sub-element $K_i$  with $i = 1$ or $2$. However, the proof of Lemma \ref{ltrace} is also applicable to this weighting. We use the specific definition \eqref{kappa} of $\kappa_i$ because the proof of the trace inequality \eqref{ineq_main}, and likewise of the trace inequalties, \eqref{convex1} and \eqref{convex2} is of independent interest in its own right. Furthermore, the presence of $\kappa_i$ is essential to keeping the constant in \eqref{ineq_main} independent of the location of the interface relative to the mesh.

\section{Error Analysis}\label{sec:error}
\setcounter{equation}0

\subsection{Boundedness and stability of $B_h(\cdot,\cdot)$}
To consider the boundedness and stability of the primal forms $B_h^{(i)}(\cdot,\cdot)$, we define the following semi-norms and norms for $v\in V(h)$:
\begin{eqnarray}
&&|v|^2_{1,\Omega_1\cup\Omega_2} =  \|\alpha(x)^{1/2}\nabla v\|^2_{L^2(\Omega)}, \nonumber\\
&&|v|^2_{0,\mathcal{E}_h^\Gamma} = \sum_{K\in \mathcal{T}_h^\Gamma} \eta_\beta h_K^{-1}\|[v]\|^2_{L^2(e_K)},\quad
|v|^2_{0,\mathcal{T}_h^\Gamma} = \sum_{K\in\mathcal{T}_h^\Gamma}\eta\|\alpha(x)^{1/2}r_{e_K}([v])\|_{L^2(K)}^2 ,\nonumber\\
&&\|v\|^2_{B_h^{(1)}} = |v|^2_{1,\Omega_1\cup\Omega_2} + |v|^2_{0,\mathcal{E}_h^\Gamma}+ \sum_{K\in \mathcal{T}_h^\Gamma}\eta_\beta^{-1}
h_K \|\{\alpha(x) \nabla v\}\|^2_{L^2(e_K)},\label{norm}\\
&&\|v\|^2_{B_h^{(2)}} = |v|^2_{1,\Omega_1\cup\Omega_2} + |v|^2_{0,\mathcal{E}_h^\Gamma}+|v|^2_{0,\mathcal{T}_h^\Gamma},\qquad \mbox{with}~~ \beta=1.\label{norm2}
\end{eqnarray}
Here, $e_K = K\cap \Gamma$.

\begin{lemma}[Boundedness of $B_h^{(1)}(\cdot,\cdot)$] We have
\begin{equation}\label{bound}
B_h^{(1)}(w,v)\leq C_b \|w\|_{B_h^{(1)}}\,\|v\|_{B_h^{(1)}},
\quad\forall\, w,v \in V(h),
\end{equation}
where $C_b$ is a positive constant dependent only on $\beta$. Actually, we can make this upper-bound constant explicit with $C_b=(\sqrt{\beta^2+2\beta+2}+\sqrt{\beta^2-2\beta+2})/2$.
\end{lemma}
\begin{proof}
The inequality \eqref{bound} is a direct consequence of the definitions \eqref{norm} and the Cauchy-Schwarz inequality.
\end{proof}

Notice that norm \eqref{norm} is the natural choice for obtaining the boundedness of the bilinear form $B_h^{(1)}(\cdot,\cdot)$ in $V(h)$, whereas the similar continuity of $B_h^{(2)}(\cdot,\cdot)$ in norm \eqref{norm2} is only valid in the discrete space $V_h^\Gamma$, which is also a simple result of the Cauchy-Schwarz inequality.

The following lemma demonstrates the coercivity of $B_h^{(i)}(\cdot,\cdot)$ in its respective norm $\|\cdot\|^2_{B_h^{(i)}}$. Note that the second part of the results shows that scheme \eqref{scheme2} is ``parameter-friendly.''

\begin{lemma}[Stability of $B_h^{(i)}(\cdot,\cdot)$]\label{lem_stab}  There exists a constant $C_s^{(1)}>0$ such that
\begin{equation}
B_h^{(1)}(v,v)\geq C_s^{(1)} \|v\|^2_{B_h^{(1)}}, \quad \forall\, v \in V_h^\Gamma,
\label{stable}
\end{equation}
provided the penalty parameter $\eta_\beta$ is chosen sufficiently large. Moreover, if $\eta_1\geq 1$ and $\eta \geq 2$, there exists a constant $C_s^{(2)}>0$ such that
\begin{equation}
B_h^{(2)}(v,v)\geq C_s^{(2)} \|v\|^2_{B_h^{(2)}}, \quad \forall\, v \in V_h^\Gamma.
\label{stable2}
\end{equation}
Here, $C_s^{(i)}$ is a positive constant dependent only on the parameter $\eta_{\beta}$ or on the parameters $\eta_1$ and $\eta$.
\end{lemma}
\begin{proof}
We perform an analysis of the symmetric and non-symmetric variants of $B_h^{(1)}(\cdot,\cdot)$, that is, $\beta=\pm 1$. For the alternative $\beta$ in $B_h^{(1)}(\cdot,\cdot)$, the only difference in the stability analysis is the selection of parameter $\eta_\beta$ and the determination of the corresponding $C_s^{(1)}$. We omit the details.

For $\beta=1$, by the Cauchy-Schwarz inequality, we know that
\begin{align*}
\int_{e_K}\{\alpha(x) \nabla v\}\cdot [v]  \le
h_K^{\frac12} \|\{\alpha(x) \nabla v\}\|_{L^2(e_K)}\cdot h_K^{-\frac12}\|[v]\|_{L^2(e_K)}.
\end{align*}
Using the inequality $2ab\le\epsilon a^2+\frac{1}{\epsilon} b^2$, we deduce that
\begin{align*}
B_h^{(1)}(v,v) &=  |v|^2_{1,\Omega_1\cup\Omega_2} + |v|_{0,\mathcal{E}_h^\Gamma}^2 - 2\int_{\Gamma}\{\alpha(x) \nabla v\}\cdot [v] \\
&\geq  |v|^2_{1,\Omega_1\cup\Omega_2} + |v|_{0,\mathcal{E}_h^\Gamma}^2 - 2\sum_{K\in \mathcal{T}_h^\Gamma}h_K^{\frac12} \|\{\alpha(x) \nabla v\}\|_{L^2(e_K)}\cdot h_K^{-\frac12}\|[v]\|_{L^2(e_K)} \\
& \geq |v|^2_{1,\Omega_1\cup\Omega_2} + |v|_{0,\mathcal{E}_h^\Gamma}^2 - 2\left(\epsilon\sum_{K\in \mathcal{T}_h^\Gamma}\eta_1^{-1}h_K \|\{\alpha(x) \nabla v\}\|_{L^2(e_K)}^2 + \frac{|v|_{0,\mathcal{E}_h^\Gamma}^2}{4\epsilon}\right)\\
& = |v|^2_{1,\Omega_1\cup\Omega_2} + (1- \frac{1}{2\epsilon}) |v|_{0,\mathcal{E}_h^\Gamma}^2 - 2\epsilon\sum_{K\in \mathcal{T}_h^\Gamma}\eta_1^{-1}h_K \|\{\alpha(x) \nabla v\}\|_{L^2(e_K)}^2,
\end{align*}
where $\epsilon>0$ is an arbitrary constant number. To estimate the last term, we draw on Lemma \ref{ltrace} and obtain
\begin{eqnarray}
\sum_{K\in \mathcal{T}_h^\Gamma}h_K \|\{\alpha(x) \nabla v\}\|_{L^2(e_K)}^2\le 2 C \max_{x\in\Omega}\{\alpha(x)\}  |v|^2_{1,\Omega_1\cup\Omega_2}.\label{critical_est}
\end{eqnarray}
For any $\epsilon>1$, if we choose $\eta_1> 8\epsilon C\max_{x\in\Omega}\{\alpha(x)\}  $, we obtain
\begin{align}\label{stable0}
B_h^{(1)}(v,v) &\geq \frac12(|v|^2_{1,\Omega_1\cup\Omega_2} + |v|_{0,\mathcal{E}_h^\Gamma}^2) \ge C_s^{(1)} \|v\|^2_{B_h^{(1)}}.
\end{align}
This completes the proof for the case $\beta=1$. For $\beta=-1$, the result of \eqref{stable} follows from the identity $B_h^{(1)}(v, v)=|v|^2_{1,\Omega_1\cup\Omega_2} + |v|_{0,\mathcal{E}_h^\Gamma}^2$.

Concerning the second formulation, we observe that
$$-\int_{e} \{\alpha(x) \nabla v\}\cdot [v] = \int_{K} \alpha(x) \nabla v\cdot r_e([v]),$$
for any $e=K\cap\Gamma$ and that
\begin{align*}
B_h^{(2)}(v,v)&=\|v\|^2_{B_h^{(2)}}+2\sum_{e\in\mathcal{E}_h^\Gamma}\int_{\Omega}\alpha(x)\nabla v\cdot r_e([v]) \\
&\ge(1-\epsilon)|v|^2_{1,\Omega_1\cup\Omega_2} +(1-\frac{1}{\epsilon\eta})|v|^2_{0,\mathcal{T}_h^\Gamma}+|v|^2_{0,\mathcal{E}_h^\Gamma}.
\end{align*}
Then, \eqref{stable2} holds with $C_s^{(2)}=\min(1-\epsilon, 1-\frac{1}{\epsilon\eta})$. In particular, by choosing $\epsilon = \frac{1}{\sqrt{2}}$, and $\eta \geq 2$, we have $C_s^{(2)} = 1- \frac{1}{\sqrt{2}}$.
\end{proof}

\subsection{Approximation capability of $V_h^\Gamma$}
We want to show that the XFE space has optimal approximation quality for piecewise smooth functions $w\in H^p(\Omega_1\cup \Omega_2)$.
For this purpose, we construct an interpolant of $w$ by the nodal interpolants of $H^s$-extensions of $w_1$ and $w_2$ as follows. Let $s\ge 2$ be an integer and choose extension operators $E_i: H^s(\Omega_i)\mapsto H^s(\Omega)$ such that
\[(E_i w)|_{\Omega_i}=w \quad\text{ and }\quad \|E_i w\|_{H^s(\Omega)}\lesssim\|w\|_{H^s(\Omega_i)},\quad  i=1, 2.\]

Let $I_h$ be the standard nodal interpolation that is associated with $V_h$ and that satisfies (\cite{schwab98})
 \begin{equation}
   \|v-I_h v\|_{H^j(\Omega)}\le Ch^{\mu-j}\|v\|_{H^{s}(\Omega)}, \quad j=0,1,2,\label{standard_interpolation}
 \end{equation}
where $v\in H^s(\Omega)\cap H^1_0(\Omega)$ with $s\ge 2$ and $\mu=\min\{p+1, s\}$. Denote $E_i w$ by $\tilde{w}_i$, and define an interpolation of $w\in V$ to $V_h^\Gamma$ by
\begin{equation}\label{interp}
\Pi_h w = \chi_1 I_h \tilde{w}_1 + \chi_2 I_h \tilde{w}_2.
\end{equation}
We present an approximation error bound for the XFE space:
\begin{align}
\|w-\Pi_h w\|_{B_h^{(1)}}^2 + \|w-\Pi_h w\|_{B_h^{(2)}}^2 \lesssim h^{2(\mu-1)}|w|^2_{H^s(\Omega_1\cup\Omega_2)}.\label{interp_est}
\end{align}
For the proof of this result, we need to address the interpolation error along the interface.
Indeed, we can apply Lemma \ref{ltrace2} and obtain
\begin{align}
h_K \|\{\alpha(x)\nabla (w-\Pi_h w)\}\|^2_{L^2(e_K)}\lesssim
 & \sum_{i=1,2} h_K \|\chi_i\nabla (w-I_h w)\|^2_{L^2(e_K)}
= \sum_{i=1,2} h_K \|\nabla (\tilde{w}_i-I_h \tilde{w}_i)\|^2_{L^2(e_K)} \nonumber\\
\leq & \sum_{i=1,2} \left( \|\tilde{w}_i-I_h \tilde{w}_i\|^2_{H^1(K)}  + h_K^2 \|\tilde{w}_i-I_h \tilde{w}_i\|^2_{H^2(K)}\right)\label{item1}
\end{align}
and
\begin{align}
h_K^{-1}\|[w-\Pi_h w]\|^2_{L^2(e_K)} \leq & \sum_{i=1,2} h_K^{-1} \|\chi_i(w-I_h w)\|^2_{L^2(e_K)}
= \sum_{i=1,2} h_K^{-1} \|\tilde{w}_i-I_h \tilde{w}_i\|^2_{L^2(e_K)} \nonumber\\
\leq & \sum_{i=1,2} \left( \|\tilde{w}_i-I_h \tilde{w}_i\|^2_{L^2(K)}  + h_K^2 \|\tilde{w}_i-I_h \tilde{w}_i\|^2_{H^1(K)}\right).\label{item2}
\end{align}
Furthermore, we need the following property of the local lifting operator $r_e$ in order to address $|w-\Pi_h w|_{0,\mathcal{E}_h^\Gamma}$. We note that the reverse of the following inequality is not generally true, in particular, when one of sub-elements of $K$ degenerates.

\begin{lemma} There exists a positive constant $C$ independent of the relative position of $\Gamma$ respect to $K$ such that
\begin{align}
&\|r_e(\boldsymbol{q})\|_{L^2(K)} \le C h_K^{-\frac12}\|\boldsymbol{q}\|_{L^2(e)},\,\, \forall\, \boldsymbol{q}\in [L^2(e)]^2 \label{lift_est}
\end{align}
for each $e = K\cap \Gamma \in \mathcal{E}_h^\Gamma$.
\end{lemma}
\begin{proof}
We take $\boldsymbol{w}_h=r_e(\boldsymbol{q})$ in \eqref{liftop} and find
$$\|\alpha(x)^{1/2}r_e(\boldsymbol{q})\|_{L^2(K)}^2\le\|\boldsymbol{q}\|_{L^2(e)}\|\{\alpha(x)r_e(\boldsymbol{q})\}\|_{L^2(e)}\le C h_K^{-\frac12}\|\boldsymbol{q}\|_{L^2(e)}\|r_e(\boldsymbol{q})\|_{L^2(K)},$$
where the last inequality follows from \eqref{ineq_main}.
\end{proof}

From \eqref{item1}, \eqref{item2}, and \eqref{lift_est}, the estimates of the edge terms are reduced to those of bulk terms, which then follow the standard interpolation arguments \eqref{standard_interpolation}. We point out that \eqref{trace} can be modified by replacing $K$ in the right-hand side by its larger sub-element $K_i$ with $i=1$ or $2$. Thus, the alternative definition of $\kappa_i$ from \eqref{kappa} is possible \cite{hansbo05, wu2010} for many other choices. We emphasize that \eqref{trace} leads to a uniform constant hidden in $\lesssim$ of the interpolation estimates \eqref{interp_est}.

\subsection{Error estimates}

To summarize, we have the following error estimate for each scheme in its respective norm.
\begin{theorem}\label{thm1}
Assume that the interface $\Gamma$ is $C^2$ smooth and that the solution of the elliptic interface problem \eqref{interface_problem} satisfies $u\in H^s(\Omega_1\cup\Omega_2)$, where $s\ge 2$ is an integer. Let $\mu= \min\{p+1,s\}$. The following error estimates hold for any $h\in (0, h_0]$.
  \begin{enumerate}
    \item[{\rm(i)}] If $\eta_\beta$ is chosen sufficiently large (see \eqref{stable}) and $u_h$ is the solution to the first scheme of \eqref{dg_primal}, then
    \begin{equation}\label{errorcs}
    \|u-u_h\|_{B_h^{(1)}} \lesssim h^{\mu-1}
    \|u\|_{H^s(\Omega_1\cup\Omega_2)},\quad \forall\,0<h\le h_0.
    \end{equation}
    \item[{\rm(ii)}] For any given $\eta_1\geq 1$ and $\eta \geq 2$ with $u_h$ as the solution to the second scheme of \eqref{dg_primal}, we have
    \begin{equation}\label{errorcs2}
    \|u-u_h\|_{B_h^{(2)}} \lesssim  h^{\mu-1}
    \|u\|_{H^s(\Omega_1\cup\Omega_2)},\quad \forall\,0<h\le h_0.
    \end{equation}
  \end{enumerate}
The hidden constants in the above estimates are dependent on the angle condition of the mesh $\mathcal{T}_h$, the degree of the polynomials, the parameter in the scheme, and $\alpha(x)$, but are independent of the location of the interface relative to the mesh. Here, the constant $h_0$ is from Lemma~\ref{ltrace}.

\end{theorem}
\begin{proof}
Let $\Pi_h u\in V_h^\Gamma$ be the interpolant of $u$ as defined in \eqref{interp}. We recall the stability \eqref{stable} and \eqref{stable2} of the bilinear form $B_h^{(i)}(\cdot, \cdot)$. Denote by $B_h(\cdot,\cdot) = B_h^{(i)} (\cdot,\cdot)$ and $C_s=C_s^{(i)}$ with $i = 1, 2$, and we have
\begin{equation}
C_s\|\Pi_h u-u_h\|_{B_h}^2 \le B_h(\Pi_h u-u_h,\Pi_h u-u_h) = B_h(\Pi_h u-u,\Pi_h u-u_h),
\label{err1}
\end{equation}
where we use the Galerkin orthogonality \eqref{dg_ortho} to derive the last identity.

The error estimate for the first scheme follows from the boundedness \eqref{bound} of $B_h^{(1)}(\cdot, \cdot)$ and the triangle inequality
\begin{equation}
\|u-u_h\|_{B_h^{(1)}} \le\|u-\Pi_h u\|_{B_h^{(1)}}+\|\Pi_h u-u_h\|_{B_h^{(1)}} \leq (1+C_b/C_s^{(1)}) \| u-\Pi_h u\|_{B_h^{(1)}}.
\end{equation}
Thus, \eqref{errorcs} is the consequence of \eqref{interp_est}.

To derive the error estimate for the second scheme, we observe that for $w\in V$ and $v_h\in V_h^\Gamma$,
\begin{align}
B_h^{(2)}(w,v_h)=&\int_{\Omega_1\cup\Omega_2}\alpha(x) \nabla w \cdot \nabla v_h
-\int_{\Gamma} \{\alpha(x)\nabla w\}\cdot [v_h]
+\sum_{e\in\mathcal{E}_h^\Gamma}\int_{\Omega}\alpha(x)\nabla v_h \cdot r_e([w])\nonumber\\
&+\sum_{K\in\mathcal{T}_h^\Gamma}\frac{\eta_1}{h_K}\int_{K\cap\Gamma}[w]\cdot [v_h]+\sum_{e\in\mathcal{E}_h^\Gamma}\int_{\Omega}\eta\alpha(x) r_e([w])\cdot r_e([v_h])\nonumber\\
\le&C \left(\|w\|_{B_h^{(1)}}^2+|w|^2_{0,\mathcal{T}_h^\Gamma}\right)^{\frac12}\|v_h\|_{B_h^{(2)}}.\label{bound2*}
\end{align}
Instead of using the boundedness of the bilinear form on $V(h)$, which is not generally true for $B_h^{(2)}(\cdot,\cdot)$ in the norm $\|\cdot\|_{B_h^{(2)}}$, we substitute $w=\Pi_h u-u$ and $v_h=\Pi_h u-u_h$ in \eqref{bound2*} and obtain that
\begin{align}
\|u-u_h\|_{B_h^{(2)}}\le C\left(\|u-\Pi_h u\|_{B_h^{(1)}}+\|u-\Pi_h u\|_{B_h^{(2)}}\right).\label{inf-sup}
\end{align}
Then, the proof is completed from the interpolation error bound \eqref{interp_est}.
\end{proof}

We derive the optimal order $L^2$-error estimate for the first scheme when $\beta=1$ by using Nitsche's duality argument (cf. \cite{ciarlet78}). The $L^2$-error estimate for the second scheme follows a similar procedure plus a variant of \eqref{bound2*}. We omit the details.

Consider an auxiliary function $w$ as the solution to the adjoint problem
\begin{equation}\label{ePD}\left\{
\begin{aligned}
            & - \nabla\cdot\big(\alpha(x) \nabla w\big)  =  u-u_h,   &\text{ in }\Omega_1\cup\Omega_2,\\
            & [w]=0, \quad  [\alpha(x) \nabla w]=0, &\text{ on } \Gamma, \\
            &  w =  0, &\text{ on } \partial\Omega.
\end{aligned}\right.
\end{equation}
As $\Omega$ is convex, elliptic regularity gives (cf. \cite{b70})
\begin{equation}\label{ewsta}
\|w\|_{H^2(\Omega_1\cup\Omega_2)}\lesssim \|u-u_h\|_{L^2(\Omega)}.
\end{equation}

 \begin{theorem}\label{thml2} Under the conditions of Theorem~\ref{thm1}, the following estimate holds for the first scheme when $\beta=1$:
 \begin{equation*}
    \|u-u_h\|_{L^2(\Omega)}\lesssim h^{\mu}
    \|u\|_{H^s(\Omega_1\cup\Omega_2)},\quad \forall\,0<h\le h_0.
  \end{equation*}
 \end{theorem}
\begin{proof}Let $\theta=u-u_h$. Testing \eqref{ePD} by $\theta$ and using \eqref{dg_ortho}, we obtain
\begin{align}\label{et21}
    \|\theta\|_{L^2(\Omega)}^2=B_h^{(1)}(w,\theta)=B_h^{(1)}(\theta,w)=B_h^{(1)}(\theta,w-\Pi_h w),
\end{align}
where $\Pi_h w\in V_h^\Gamma$ satisfies the estimate (cf. \eqref{interp_est})
\begin{equation}\label{et22}
    \|w-\Pi_h w\|_{B_h^{(1)}}\lesssim  h|w|_{H^2(\Omega_1\cup\Omega_2)}.
\end{equation}
  Therefore, from \eqref{bound} and \eqref{ewsta},
\begin{align*}
    \|\theta\|_{L^2(\Omega)}^2\le C_b\|w-\Pi_h w\|_{B_h^{(1)}}\|\theta\|_{B_h^{(1)}}\lesssim h\|\theta\|_{L^2(\Omega)}\|\theta\|_{B_h^{(1)}},
\end{align*}
that is $\|\theta\|_{L^2(\Omega)}\lesssim h\|\theta\|_{B_h^{(1)}}$, which by using \eqref{errorcs} completes the proof of Theorem~\ref{thml2}.
\end{proof}

\section{Numerical examples}\label{sec:numerics}
\setcounter{equation}0

To test the numerical methods, we consider the following example. Let domain $\Omega$ be the unit square $(0,1)\times(0,1)$ and interface $\Gamma$ be the zero level set of the function $\phi(x)=(x_1-0.5)^2+(x_2-0.5)^2-1/8$ so that the subdomain $\Omega_1$ is characterized by $\phi(x)<0$ and $\Omega_2$ by $\phi(x)>0$. We use the Cartesian grids to partition the domain $\Omega$ into squares of the same size $h$. Let the exact solution be
\begin{align*}
u(x)=\left\{\begin{array}{ll}
			 1/\alpha_1\exp(x_1x_2),& x\in\Omega_1,\\
			 1/\alpha_2\sin(\pi x_1)\sin(\pi x_2),&x\in\Omega_2.	
			\end{array}\right.
\end{align*}
The right-hand side can be computed accordingly.

We examine the $h$-convergence rate of the first numerical scheme with $\beta=1$, that is, the symmetric case, and choose the parameter $\eta_1=20$ in all cases.
Theorems~\ref{thm1} and \ref{thml2} imply that
\begin{align*}
|u-u_h|_{1,\Omega_1\cup\Omega_2}=\left(\sum_{i=1}^2\|\alpha(x)^{1/2}\nabla (u-u_h)\|_{L^2(\Omega_i)}^2\right)^{1/2}\lesssim C h^p,\qquad
\|u-u_h\|_{L^2(\Omega)}\lesssim C h^{p+1}.
\end{align*}

\begin{figure}[htp]
\centering
\includegraphics[scale=1]{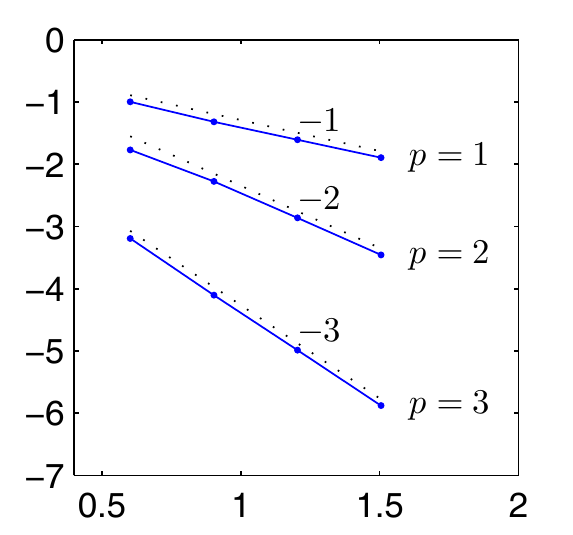}
\includegraphics[scale=1]{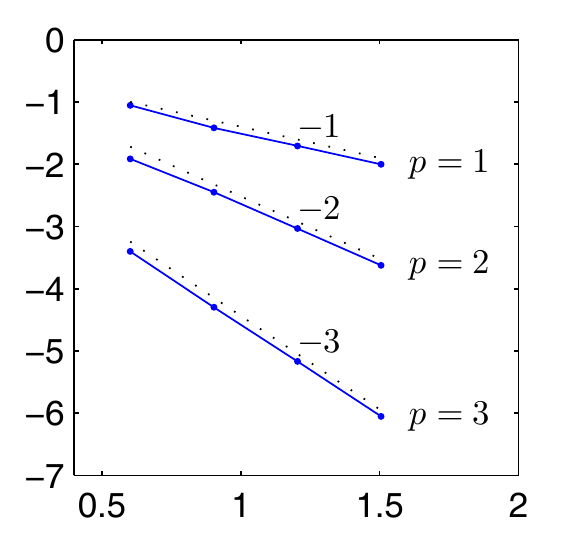}
\caption{Error Reduction of $\log_{10}\left(|u-u_h|_{1,\Omega_1\cup\Omega_2}/|u|_{1,\Omega_1\cup\Omega_2}\right)$ to $\log_{10}(1/h)$. \newline
Left: $\alpha_1=10, \alpha_2=1$. Right: $\alpha_1=1, \alpha_2=10$.
}\label{fe1}
\end{figure}

\begin{figure}[htp]
\centering
\includegraphics[scale=1]{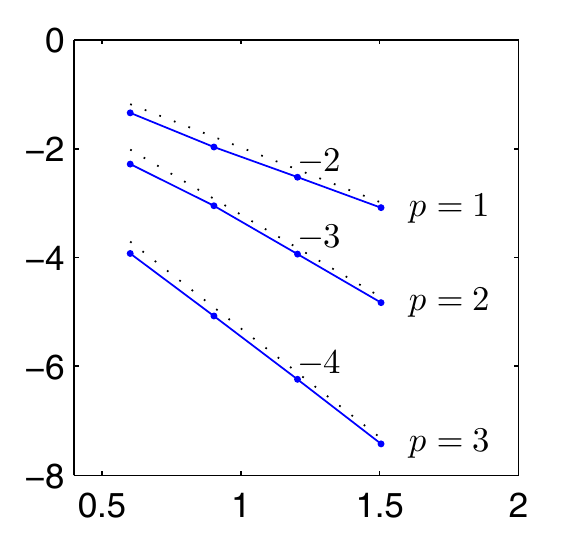}
\includegraphics[scale=1]{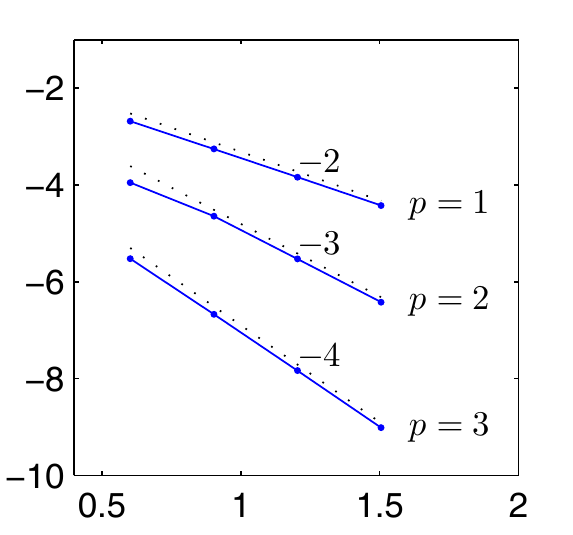}
\caption{Error Reduction of $\log_{10}\left(\|u-u_h\|_{L^2(\Omega)}/\|u\|_{L^2(\Omega)}\right)$ to $\log_{10}(1/h)$. \newline
Left: $\alpha_1=10, \alpha_2=1$. Right: $\alpha_1=1, \alpha_2=10$.
}\label{fe2}
\end{figure}
Figure \ref{fe1} (left) plots $\log_{10}\left(|u-u_h|_{1,\Omega_1\cup\Omega_2}/|u|_{1,\Omega_1\cup\Omega_2}\right)$ versus $\log_{10}(1/h)$ with $h=1/4, 1/8,$ $ 1/16, 1/32$ for $\alpha_1=10$, $ \alpha_2=1$, and $p=1,2,3$, respectively. Figure \ref{fe1} (right) gives corresponding plots for $\alpha_1=1$, $ \alpha_2=10$. The dotted lines give reference lines of slopes $-1, -2$, and $-3$, respectively. Figure \ref{fe2} shows the results on the relative errors in the $L^2$-norm for both choices of the coefficient $\alpha(x)$. The convergence rate of $O(h^p)$ and $O(h^{p+1})$ are observed, respectively, in these cases, which confirms our theoretical results.

\end{document}